\documentclass[10pt]{amsart}

\usepackage[margin=1.25in]{geometry} 
\usepackage[table]{xcolor}
\usepackage{amsfonts}
\usepackage{amsmath,amsthm,amssymb,mathtools, scrextend,enumerate,bbm,tikz-cd,verbatim}
\usepackage{mathrsfs}
\usepackage{tikz-cd}
\usepackage[utf8]{inputenc}

\usepackage{fancyhdr}

\newtheorem{thm}{Theorem}[section]
\newtheorem{lem}[thm]{Lemma}
\newtheorem{prop}[thm]{Proposition}

\theoremstyle{definition}

\newtheorem{rem}[thm]{Remark}

\newtheorem*{defn*}{Definition}
\newtheorem*{thm*}{Theorem}
\newtheorem*{prop*}{Proposition}
\newtheorem*{cor*}{Corollary}

\newtheorem*{question*}{Question}

\newcommand{\R}{\mathbb{R}}

\DeclareMathOperator{\Index}{index}
\DeclareMathOperator{\nullity}{nullity}
\newcommand{\nul}{\operatorname{nul}}

\begin{document}
 

\title{On the index of minimal surfaces with free boundary in a half-space}

\address{Department of Mathematics, Stanford University, 450 Jane Stanford Way, Bldg
380, Stanford, CA 94305}

\author{Shuli Chen}
\email{shulic@stanford.edu}

\begin{abstract}
We study the Morse index of minimal surfaces with free boundary in a half-space. We improve previous estimates relating the Neumann index to the Dirichlet index and use this to answer a question of Ambrozio, Buzano, Carlotto, and Sharp concerning the non-existence of index two embedded minimal surfaces with free boundary in a half-space. We also give a simplified proof of a result of Chodosh and Maximo concerning lower bounds for the index of the Costa deformation family. 
\end{abstract}

 \maketitle

\section{Introduction}
Given an orientable Riemannian $3$-manifold $M^3$, a minimal surface is a  critical point of the area functional. For minimal surfaces in $\R^3$, the maximum principle implies that they must be non-compact.
Hence, minimal surfaces in $\R^3$ are naturally studied under some weaker finiteness assumption,
such as finite total curvature or finite Morse index. We use the word \emph{bubble} to denote a complete, connected, properly embedded minimal surface of finite total curvature in $\R^{3}$. 
As shown by Fischer-Colbrie \cite{fischer1985complete} and Gulliver--Lawson \cite{gulliver1986structure,gulliver1986index},  a complete oriented minimal surface in $\R^3$ has finite index if and
only if it has finite total curvature. In particular, bubbles have finite Morse index.
Bubbles of low index are classified: the plane is the only bubble of index 0 \cite{fischer1980structure,do1979stable,pogorelov1981stability}, the catenoid is the only bubble of index 1 \cite{lopez1989complete}, and there is no bubble of index 2 or 3 \cite{chodosh2016topology,chodosh2018topology}.

Similarly, given an orientable Riemannian $3$-manifold $M^3$ with boundary $\partial M$, a free boundary minimal surface is a critical point of the area functional
among all submanifolds with boundaries in $\partial M$. We use the word \emph{half-bubble} to denote a complete, connected, properly embedded minimal surface of finite curvature that is contained in a half-space of $\R^3$ and has non-empty free boundary with respect to the boundary of this half-space.
Half-bubbles arise as blow-up limits of certain sequences of free boundary minimal surfaces in compact Riemannian $3$-manifolds with boundary, as studied in \cite{ambrozio2018compactness,ambrozio2019bubbling}.

Using reflection across the bounding plane we obtain a symmetric bubble from a half-bubble, so results on the index of bubbles allow us to obtain bounds for the index of half-bubbles. In particular, half-bubbles have finite Morse index.
In \cite{ambrozio2019bubbling}, the authors introduced the notion of Morse index with Dirichlet conditions for a half-bubble, and observed that the usual (Neumann) Morse index is lower bounded by the Dirichlet index. This implies the Morse index of a half-bubble is at least half the index of the corresponding symmetric bubble. Using this, they showed the half-plane is the only half-bubble of index 0 and the half-catenoid is the only half-bubble of index 1. They then asked whether there exists a half-bubble of index 2.

In this paper, we observe that the (Neumann) index of a half-bubble can be bounded from below by the \emph{sum} of the Dirichlet index and the Dirichlet nullity. Moreover, the Dirichlet nullity is always at least one-dimensional. 
Therefore we obtain an improved inequality between the index of a half-bubble and the corresponding symmetric bubble:
\begin{thm}\label{thm: index comparison}
Let $\Sigma$ be a half-bubble that is neither a half-plane or a half-catenoid, and let $\check{\Sigma}$ be the corresponding symmetric bubble.
Then 
$$\Index(\Sigma) \le \Index(\check{\Sigma}) \le 2\Index(\Sigma) -1.$$
If, in addition, all the ends of $\check{\Sigma}$ are orthogonal to the bounding plane, then we have
$$\Index(\Sigma) + 1\le \Index(\check{\Sigma}) \le 2\Index(\Sigma) -2.$$
\end{thm}
Using this, we conclude
\begin{thm}\label{thm: no index two}
There doesn't exist two-dimensional half-bubbles whose Morse index equals two.
\end{thm}
\begin{proof}
If there exists a half-bubble of index two, then by Theorem~\ref{thm: index comparison}, the index of the symmetric bubble $\check{\Sigma}$ is bounded by 
$$2 = \Index(\Sigma) \le \Index(\check{\Sigma}) \le 2\Index(\Sigma) -1 = 3.$$

However, by \cite[Theorem 2]{chodosh2016topology} and \cite[Theorem 1.7]{chodosh2018topology}, there doesn't exist a bubble in $\R^3$ of index two or three.
\end{proof}

As a result, the geometric convergence results and blow-up analysis in \cite[Theorem 8 and Theorem 9]{ambrozio2019bubbling} can be extended to sequences of free boundary minimal surfaces of index at most two.  We also note that our analysis of the index and nullity with Dirichlet and Neumann boundary conditions shares some features in common with the arguments in \cite{kapouleas2019index}.

Using the same observation, we also give a shorter proof of an result of \cite{chodosh2018topology} which states the Costa family has Morse index at least four. We further discuss about the index of complete embedded free boundary minimal surface contained in a quarter-space of $\R^3$.

\subsection{Acknowledgments}
I am grateful to my advisor, Otis Chodosh, for bringing this problem to my attention, for his continuous support and encouragement, and for many helpful comments on earlier drafts of this paper. I also want to thank Davi Maximo for his interest in this work. 

\section{Results on indices of bubbles and half-bubbles}
As in the introduction, we use the word bubble to denote a complete, connected, properly embedded minimal surface of finite total curvature in $\R^{3}$; we use the word half-bubble to denote a complete, connected, properly embedded minimal surface of finite curvature that is contained in a half-space of $\R^3$ and has non-empty free boundary with respect to the boundary of this half-space.

Let $\Pi \subset \R^3$ be a plane and $\Pi_{-}$, $\Pi_{+}$ be the closed half-spaces bounded by $\Pi$. If $X^2 \subset \R^3$ is a bubble symmetric under the reflection across $\Pi$, then $\Sigma : = X \cap \Pi_-$ is a half-bubble. Conversely, if $\Sigma^2 \subset \Pi_-$ is a half-bubble, then by reflecting $\Sigma$ across $\Pi$ we get a minimal surface $\check{\Sigma}$ without boundary, which is smooth by standard elliptic theory. Thus $\check{\Sigma}$ is a bubble symmetric under reflection across $\Pi$. 

As shown in \cite{ambrozio2019bubbling}, for the symmetric bubble $\check{\Sigma}$ we have that either 
\begin{enumerate}[(a)]
\item each of its ends is contained in one of the half-spaces $\Pi_-$ and $\Pi_+$ (provided we remove from $\check{\Sigma}$ a sufficiently large ball centered at the origin) and is parallel to $\Pi$, or 
\item all of the ends intersect $\Pi$ and are orthogonal to $\Pi$.
\end{enumerate}
We will refer to these two cases as case (a) and (b), respectively. We will say a half-bubble satisfies case (a) or (b) if the corresponding symmetric bubble satisfies case (a) or (b), respectively.

\subsection{Index and nullity of a bubble}
For a bubble $X \subset \R^3$, we consider the Jacobi operator $L := - \Delta - |\nabla N|^2$ and the associated quadratic form
$$Q_{X}(f,f) := \int_X(|\nabla f|^2 - |\nabla N|^2f^2)dA,$$
where $N\colon X \to S^2$ denotes the Gauss map.
We define the Morse index of $X$, $\Index(X)$, as the largest dimension
of a linear subspace of $C^\infty_c (X)$ where $Q_{X}$ is negative definite. 

We call a $C^2$-function $f\colon X \to \R$ a \emph{Jacobi field} if it satisfies $Lf = 0$ on $X$. Let $\mathcal{K}(X)$ be the space of bounded Jacobi fields on $X$. We define the nullity of $X$, $\nul(X)$, as the dimension of $\mathcal{K}(X)$. Inside $\mathcal{K}(X)$
lies the subspace of linear functions $L(X) = \{\langle v,N\rangle \mid v \in \R^3\}$ generated by translations. 
The embeddedness of $X$ guarantees that the ends of $X$ are all parallel to some line $\ell$. Then rotation around $\ell$ 
generates a bounded Jacobi field $ \det(p, N, \lambda)$ , where $p$ is the position vector of $X$, and $\lambda$ is a unit vector on $\ell$ \cite{montiel1991schrodinger}. 
If $X$ is neither the plane nor the catenoid, then $\det(p, N, \lambda)$ is a nonzero Jacobi field not contained in $L(X)$, so $\mathcal{K}(X)$ is at least 4-dimensional in this case.

If $X$ is further assumed to be symmetric under the reflection across $\Pi$, then we can decompose functions on $X$ into even and odd functions.
We then define the odd Morse index $\Index_-(X)$ (resp. the even Morse index $\Index_+(X)$) as the largest dimension
of a linear subspace of $(C^\infty_c)_- (X)$ , the space of odd $C^\infty_c$ functions on $X$ (resp. $(C^\infty_c)_+ (X)$, the space of even $C^\infty_c$ functions on $X$) where $Q_{X}$ is negative definite. We also define the odd nullity $\nul_-(X)$ (resp. the even nullity $\nul_+(X)$) as the dimension of bounded odd (resp. even) Jacobi functions on $X.$
As shown in \cite[Lemma 16]{ambrozio2019bubbling}, 
$$\Index(X) = \Index_-(X) + \Index_+(X),$$
and it is clear that
$$\nul(X) = \nul_-(X) + \nul_+(X).$$

We can give some preliminary bounds on the even and odd nullities:
\begin{lem}\label{lem: bubble nullity}
Let $X$ be a bubble symmetric under reflection across $\Pi$, then $\nul_-(X) \ge 1$ and $\nul_+(X) \ge 2$. 

If we further assume that $X$ is neither a plane nor a catenoid, then in case (a), $\nul_-(X) \ge 1$ and $\nul_+(X) \ge 3$; in case (b),  
$\nul_-(X) \ge 2$ and $\nul_+(X) \ge 2$.
\end{lem}
\begin{proof}
After a Euclidean motion we can take $\Pi$ to be the $yz$-plane. Under the reflection $\sigma$ across $\Pi$, the normal vector field $N(x) = (n_1,n_2,n_3)$ on $X$ becomes $N(\sigma(x)) = (-n_1,n_2,n_3)$. Thus $\langle e_1, N \rangle$ is an odd bounded Jacobi field while $\langle e_2, N \rangle$, $\langle e_3, N \rangle$ are even bounded Jacobi fields. This shows $\nul_-(X) \ge 1$ and $\nul_+(X) \ge 2$.

If we further assume that $X$ is neither a plane nor a catenoid, then $\phi: =\det(p, N, \lambda)$ is also a nonzero bounded Jacobi field, where $p$ is the position vector of $X$, and $\lambda$ is a unit vector parallel to the ends.  

In case (a), $\lambda$ is orthogonal to $\Pi$, and we can take it to be $e_1$. Then 
\begin{align*}
 (\phi \circ \sigma) (x) &= \det(p (\sigma(x)), N (\sigma(x)) ,e_1) \\
 &= \begin{vmatrix}-x_1 & x_2 & x_3 \\-n_1 & n_2 & n_3 \\ 1 & 0 & 0 \end{vmatrix} \\
 &=  \begin{vmatrix}x_1 & x_2 & x_3 \\n_1 & n_2 & n_3 \\ 1 & 0 & 0 \end{vmatrix} \\
 & = \det(p(x), N(x),e_1) = \phi(x),   
\end{align*}
so $\phi$ is an even Jacobi field. Therefore $\nul_-(X) \ge 1$ and $\nul_+(X) \ge 3$.

In case (b), $\lambda$ is parallel to $\Pi$. After a rotation around the $x$-axis we can take $\lambda$ to be $e_3$. Then 
\begin{align*}(\phi \circ \sigma) (x) &= \det(p (\sigma(x)), N (\sigma(x)) ,e_3) \\
&= \begin{vmatrix}-x_1 & x_2 & x_3 \\-n_1 & n_2 & n_3 \\0  & 0 & 1 \end{vmatrix} \\
& =  -\begin{vmatrix}x_1 & x_2 & x_3 \\n_1 & n_2 & n_3 \\ 0 & 0 & 1 \end{vmatrix} \\
& = -\det(p(x), N(x),e_3)  = -\phi(x),\end{align*}

so $\phi$ is an odd Jacobi field. Therefore $\nul_-(X) \ge 2$ and $\nul_+(X) \ge 2$.
\end{proof}

\subsection{Index and nullity of a half-bubble} Similarly, on a half-bubble $\Sigma \subset \Pi_-$ we can also consider the Jacobi operator $L = - \Delta - |\nabla N|^2$ and the associated quadratic form
$Q_{\Sigma}(f,f)$.
We define the Morse index of $\Sigma$, $\Index(\Sigma)$, as the largest dimension
of a linear subspace of $C^\infty_c (\Sigma)$ where $Q_{\Sigma}$ is negative definite. This agrees with the usual definition of the Morse index of a free boundary minimal surface. Notice no condition along $\partial \Sigma$ is imposed here. Because of the boundary condition we also call it the Neumann index of $\Sigma$. 
We also define the Morse index of $\Sigma$ with Dirichlet boundary conditions, $\Index_\bullet(\Sigma)$, as the largest dimension
of a linear subspace of $C^\infty_c (\mathring{\Sigma})$ where $Q_{\Sigma}$ is negative definite. Notice here we impose the Dirichlet boundary conditions along $\partial \Sigma$. We further define the (Neumann) nullity $\nul(\Sigma)$ and the Dirichlet nullity $\nul_\bullet(\Sigma)$ as the dimension of bounded solutions of $Lf = 0$ on $\mathring{\Sigma}$ with Neumann and Dirichlet boundary conditions, respectively.

By \cite[Lemma 18 and Lemma 20]{ambrozio2019bubbling}, we can relate the indices of a half-bubble and the corresponding symmetric bubble as
$$\Index(\Sigma) = \Index_+(\check{\Sigma}), \quad \Index_\bullet(\Sigma) = \Index_-(\check{\Sigma}),$$
and it is straightforward to see that 
$$\nul(\Sigma) = \nul_+(\check{\Sigma}), \quad \nul_\bullet(\Sigma) = \nul_-(\check{\Sigma}).$$
These implies
$$\Index(\check{\Sigma}) = \Index(\Sigma) +\Index_\bullet(\Sigma), \quad \nul(\check{\Sigma}) = \nul(\Sigma) +\nul_\bullet(\Sigma).$$ 
Then Lemma \ref{lem: bubble nullity} implies that
\begin{lem}\label{lem: half-bubble nullity}
For a half-bubble $\Sigma$, we have $\nul(\Sigma) \ge 2$ and  $\nul_\bullet(\Sigma) \ge 1$.

If we further assume that $\Sigma$ is neither a half-plane nor a half-catenoid, then in case (a), $\nul(\Sigma) \ge 3$ and $\nul_\bullet(\Sigma) \ge 1$; in case (b),  
$\nul(\Sigma) \ge 2$ and $\nul_\bullet(\Sigma) \ge 2$.
\end{lem}

\subsection{Bounds on the index} 
Given a half-bubble $\Sigma$ and the corresponding symmetric bubble $\check{\Sigma}$,
work of Osserman \cite{osserman1986survey} shows that $\check{\Sigma}$ is conformally equivalent to a compact Riemann surface $\overline{M}$ with finitely many punctures $p_1,\dots, p_k$, and the Gauss map $N$ extends across the punctures to a holomorphic function $\Phi\colon \overline{M} \to S^2$. The reflection symmetry of $\check{\Sigma}$ across $\Pi$ induces an anticonformal involution $\sigma$ of $\overline{M}$. The half-bubble ${\Sigma}$ is then conformal to a connected subset of $\overline{\Omega}\setminus\{p_1,\dots, p_k\}$ , whose closure in $\overline{M}$ is denoted by $\overline{\Omega}$. We denote the interior of $\overline{\Omega}$ by $\Omega$.

Endow $\overline{M}$ with any conformal Riemannian metric such that $\sigma$ is a Riemannian involution. Then we can consider the operator $L_\Phi : = -\Delta - |\nabla \Phi|^2$ and the corresponding quadratic form $Q(f,f) = \int (|\nabla f|^2 - |\nabla \Phi|^2f^2) dA$
on $\overline{M}$ and $\overline{\Omega}$. Notice that $Q$ is independent of
the particular choice of metric. We can then similarly define the various indices and nullities $\Index_\Phi(\overline{M})$,  $\nul_\Phi(\overline{M})$, $(\Index_\Phi)_-(\overline{M})$, $(\Index_\Phi)_+(\overline{M})$,  $(\nul_\Phi)_-(\overline{M})$, $(\nul_\Phi)_+(\overline{M})$, $\Index_\Phi(\overline{\Omega})$, $(\Index_\Phi)_\bullet(\overline{\Omega})$, $\nul_\Phi(\overline{\Omega})$, $(\nul_\Phi)_\bullet(\overline{\Omega})$, which are independent of the metric chosen.

By the usual variational method we see that $\Index_\Phi(\overline{M})$ equals the number of negative eigenvalues of $L$ on $\overline{M}$, and $\Index_\Phi(\overline{\Omega})$ (resp. $(\Index_\Phi)_\bullet(\overline{\Omega})$) equals the number of negative eigenvalues of $L$ on $\overline{\Omega}$ with Neumann boundary conditions (resp. Dirichlet boundary conditions).
Fischer-Colbrie \cite{fischer1985complete} showed the index and nullity of the original surface $\check{\Sigma}$ equal those of the compact Riemann surface $\overline{M}$:
$$\Index (\check{\Sigma}) = \Index_\Phi(\overline{M}), \quad \nul(\check{\Sigma}) = \nul_\Phi(\overline{M}).$$ 
Similarly we can show that 
$$(\Index_\Phi)_-(\overline{M}) = \Index_-(\check{\Sigma}), \quad (\Index_\Phi)_+(\overline{M}) = \Index_+(\check{\Sigma}),$$ 
$$(\nul_\Phi)_-(\overline{M}) = \nul_-(\check{\Sigma}), \quad (\nul_\Phi)_+(\overline{M}) = \nul_+(\check{\Sigma}),$$
$$\Index_\Phi(\overline{\Omega}) = \Index_\Phi(\Sigma), \quad (\Index_\Phi)_\bullet(\overline{\Omega}) = (\Index_\Phi)_\bullet(\Sigma),$$
$$\nul_\Phi(\overline{\Omega}) = \nul_\Phi(\Sigma), \quad (\nul_\Phi)_\bullet(\overline{\Omega}) = (\nul_\Phi)_\bullet(\Sigma).$$

Comparing Neumann and Dirichlet eigenvalues of the Jacobi operator on $\overline{\Omega}$ gives the simple bound $\Index(\overline{\Omega}) \ge  \Index_\bullet(\overline{\Omega})$. However, applying \cite[Theorem 3.3]{tran2020index} to $\overline{\Omega}$ (with $\alpha =0$, $\phi = 1$, $m = |\nabla \Phi|^2$), this inequality can be refined to

\begin{thm}[\cite{tran2020index}]\label{thm: Tran}
In the setting above,
$$\Index(\overline{\Omega}) \ge  \Index_\bullet(\overline{\Omega}) + \nul_\bullet(\overline{\Omega}).$$
\end{thm}

The full theorem of Tran actually gives an equality with a term involving eigenvalues of the Dirichlet-to-Neumann map added to the right hand side. However, this inequality form is all we need here, and from it we deduce
\begin{prop}\label{prop: main bound}
For a half-bubble $\Sigma$, $\Index(\Sigma) \ge  \Index_\bullet(\Sigma) + \nul_\bullet(\Sigma)$. 

For a symmetric bubble $X$, $\Index_+(X) \ge  \Index_-(X) + \nul_-(X)$. 
\end{prop}

By the simplicity of the first Dirichlet eigenvalue of the Jacobi operator on $\overline{\Omega}$, if $\nul_\bullet(\overline{\Omega}) \ge 2$, then we are forced to have $\Index_\bullet(\overline{\Omega}) \ge 1$. Using this observation and the discussions above we have

\begin{prop}\label{prop: nul}
For a half-bubble $\Sigma$, we have $\nul(\Sigma) \ge 2$ and $\nul_\bullet(\Sigma) \ge 1$. If $\nul_\bullet(\Sigma) \ge 2$, then $\Index_\bullet(\Sigma) \ge 1$.

For a symmetric bubble $X$, we have $\nul_+(X) \ge 2$ and $\nul_-(X) \ge 1$. If $\nul_-(X) \ge 2$, then $\Index_-(X) \ge 1$.
\end{prop}

We can now prove Theorem~\ref{thm: index comparison}:

\begin{proof}[Proof of Theorem~\ref{thm: index comparison}]
Combining the identity $$\Index(\check{\Sigma}) = \Index(\Sigma) +\Index_\bullet(\Sigma)$$
with the bound  $$\Index(\Sigma) \ge  \Index_\bullet(\Sigma) + \nul_\bullet(\Sigma)$$ in Proposition~\ref{prop: main bound}
yields
$$\Index(\Sigma) + \Index_\bullet(\Sigma)  \le \Index(\check{\Sigma}) \le 2\Index(\Sigma) -\nul_\bullet(\Sigma).$$
By Lemma~\ref{lem: half-bubble nullity}, we always have $\nul_\bullet(\Sigma) \ge 1$. If all the ends of $\check{\Sigma}$ are orthogonal to the bounding plane, then by Lemma~\ref{lem: half-bubble nullity} and Proposition~\ref{prop: nul}, we further get $\nul_\bullet(\Sigma) \ge 2$ and $\Index_\bullet(\Sigma) \ge 1$.
These bounds yield the desired result.
\end{proof}

Let $\{\Sigma_t\}_{t\ge 1}$ be the 1-parameter family of embedded genus one minimal surfaces with three ends, which is the Hoffman--Meeks deformation family of the Costa surface.
We can also provide a simplified proof of a result of \cite{chodosh2018topology}.
\begin{prop}[\cite{chodosh2018topology}]\label{cor: Costa}
$\Index(\Sigma_t) \ge 4$ for all $t$.
\end{prop}
\begin{proof}
After a Euclidean motion we can assume that $\Sigma_t$ has a reflection symmetry across the $yz$-plane. Notice that all of the ends of $\Sigma_t$ intersect the $yz$-plane, so $\Sigma_t$ satisfies case (b) here. Then Lemma \ref{lem: bubble nullity} implies that $\nul_+(\Sigma_t) \ge 2$, and $\nul_-(\Sigma_t) \ge 2$.
In this case, Proposition \ref{prop: nul} further gives $\Index_-(\Sigma_t) \ge 1$, and consequently we have
$$\Index_+(\Sigma_t) \ge \Index_-(\Sigma_t) + \nul_-(\Sigma_t) \ge 1 + 2 = 3.$$ 

Thus $\Index(\Sigma_t) = \Index_+(\Sigma_t) + \Index_-(\Sigma_t) \ge 4$.

\end{proof}

\section{Free boundary minimal surfaces in a quarter-space}
We can also consider a complete,
connected, properly embedded minimal surface $\Lambda$ that is contained in a quarter-space of $\R^3$
and has non-empty free boundary with respect to the boundary of this quarter-space,
and has finite total curvature. Following previous sections, we call it a \emph{quarter-bubble}. Denote the two bounding planes by $\Pi_1, \Pi_2$ and denote the reflections of $\R^3$ across $\Pi_1$, $\Pi_2$ by $\sigma_1$, $\sigma_2$, respectively. By reflection across $\Pi_1$, we get a free boundary minimal surface $\Sigma$ in a half-space of $\R^3$, which is smooth by standard elliptic theory. Thus $\Sigma$ is a half-bubble invariant under $\sigma_1$.
If we further reflect across $\Pi_2$, we get a bubble $X$ in $\R^3$, which is invariant under $\sigma_1$ and $\sigma_2$.

We can similarly consider the Jacobi operator $L = - \Delta - |\nabla N|^2$ (where $N\colon \Lambda \to S^2$ is again the Gauss map) and the associated quadratic form
$Q_{\Lambda}(f,f)$ on $\Lambda$ and define the notion of Morse index of $\Lambda$ with different boundary conditions as follows:
We define $\Index_{NN}(\Lambda)$ (resp. $\Index_{DN}(\Lambda)$, $\Index_{ND}(\Lambda)$, $\Index_{DD}(\Lambda)$) as the largest dimension
of a linear subspace of $C^\infty_c (\Lambda)$ (resp. $C^\infty_c (\Lambda \setminus \Pi_1)$, $C^\infty_c (\Lambda \setminus \Pi_2)$, $C^\infty_c (\mathring{\Lambda})$) where $Q_{\Lambda}$ is negative definite. 
The Morse index $\Index(\Lambda)$ of $\Lambda$ as a free boundary minimal surface equals $\Index_{NN}(\Lambda)$. 
We further define $\nul_{**}(\Sigma)$ where $* = N,D$ as the dimension of bounded solutions of $Lf = 0$ on $\mathring{\Lambda}$ with corresponding boundary conditions (Neumann or Dirichlet) on $\partial \Lambda \cap \Pi_1$, $\partial \Lambda \cap \Pi_2$. 
From definition it is clear that we have the chain of inequalites 
$$\Index_{NN}(\Lambda) \ge \Index_{DN}(\Lambda) \ge \Index_{DD}(\Lambda)$$ and $$\Index_{NN}(\Lambda) \ge \Index_{ND}(\Lambda) \ge \Index_{DD}(\Lambda).$$

For the corresponding bubble $X$ symmetric under $\sigma_1, \sigma_2$, we can decompose $C^\infty (X)$ and $C^\infty_c (X)$ into eigenspaces $(C^\infty)_{\pm \pm} (X)$, $(C^\infty_c)_{\pm \pm} (X)$.
We then define
$\Index_{\pm \pm} (X)$ as the largest dimension
of a linear subspace of $(C^\infty_c)_{\pm \pm} (X)$ where $Q_{X}$ is negative definite, and define $\nul_{\pm \pm}(X)$ as the dimension of the space of functions $f \in (C^\infty)_{\pm \pm} (X) \cap L^\infty(X)$ such that $Lf = 0$.
Arguing similarly as in \cite[Lemma 16, Lemma 18 and Lemma 20]{ambrozio2019bubbling},
we get
$$\Index(X) = \Index_{++} (X) + \Index_{+-} (X)+ \Index_{-+} (X)+\Index_{--} (X),$$
$$\Index_{NN}(\Lambda) = \Index_{++} (X),\ \Index_{ND}(\Lambda) = \Index_{+-} (X),$$
$$\Index_{DN}(\Lambda) = \Index_{-+} (X),\ \Index_{DD}(\Lambda) = \Index_{--} (X),$$
$$\Index_{NN}(\Lambda) + \Index_{DN}(\Lambda) = \Index(\Sigma),$$
$$\Index_{ND}(\Lambda) + \Index_{DD}(\Lambda) = \Index_\bullet(\Sigma),$$
and a similar decomposition for the nullities.

We can then characterize quarter-bubbles of low indices:
\begin{prop}
Let $\Lambda$ be a quarter-bubble. Then
\begin{enumerate}[(a)]
\item If $\Lambda$ has index 0, then it is a quarter of a plane. 
\item If $\Lambda$ has index 1, then it is isometric to a quarter of a catenoid.
\end{enumerate}
\end{prop}
\begin{proof}
Consider the half-bubble $\Sigma$ obtained by reflection across $\Pi_1$. Then we have the equality $$\Index_{NN}(\Lambda) + \Index_{DN}(\Lambda) = \Index(\Sigma).$$

If $\Lambda$ has index 0, then $\Index_{DN}(\Lambda) = 0$ as well, so $\Index(\Sigma) = 0$. Then by \cite[Corollary 22]{ambrozio2019bubbling}, this must be the plane.

If $\Lambda$ has index 1, then $\Index_{DN}(\Lambda) \le \Index(\Lambda) =1$, which shows $\Index(\Sigma) = 1$ or 2. By Theorem~\ref{thm: no index two}, we must have $\Index(\Sigma) = 1$. Then by \cite[Corollary 24]{ambrozio2019bubbling}, $\Sigma$ is isometric to a half-catenoid, implying $\Lambda$ is isometric to a quarter of a catenoid.
\end{proof}
\begin{rem}
Tran's result (Theorem~\ref{thm: Tran}) can be generalized to this setting with basically the same proof, and we have the four inequalities 
$$\Index_{NN}(\Lambda) \ge  \Index_{ND}(\Lambda) + \nul_{ND}(\Lambda),$$
$$\Index_{NN}(\Lambda) \ge  \Index_{DN}(\Lambda) + \nul_{DN}(\Lambda),$$
$$\Index_{ND}(\Lambda) \ge  \Index_{DD}(\Lambda) + \nul_{DD}(\Lambda),$$
$$\Index_{DN}(\Lambda) \ge  \Index_{DD}(\Lambda) + \nul_{DD}(\Lambda).$$
This shows that for the corresponding symmetric bubble $X$, we have
$$\Index_{++}(X) \ge  \Index_{+-}(X) + \nul_{+-}(X),$$
$$\Index_{++}(X) \ge  \Index_{-+}(X) + \nul_{-+}(X),$$
$$\Index_{+-}(X) \ge  \Index_{--}(X) + \nul_{--}(X),$$
$$\Index_{-+}(X) \ge  \Index_{--}(X) + \nul_{--}(X).$$
\end{rem}
As an application of these inequalities, consider the family of surfaces $\{\Sigma_t\}_{t\ge 1}$ as in Proposition \ref{cor: Costa}. After a Euclidean motion, we can assume that all of the ends of $\Sigma_t$ are parallel to the $z$-axis and $\Sigma_t$ is symmetric under reflection across $\Pi_1$ and $\Pi_2$, where $\Pi_1$ is the $xy$-plane and $\Pi_2$ is the $yz$-plane. Let $\Lambda$ be the quarter-bubble contained in the quarter-space $\{x \ge 0, y\ge 0\}$. We can then refine the bound given in Proposition \ref{cor: Costa}:
\begin{prop}
We can bound the indices and nullites of $\Sigma_t$ as follows:
\begin{table}[!ht]
\begin{tabular}{|l|l|l|l|l|}\hline
$\Sigma_t$        & $++$                   & $+-$                   & $-+$                   & $--$                   \\\hline
$\nullity$ & $\ge 1 $ & $\ge 1$ & $\ge 1$ & $\ge 1$ \\\hline
$\Index$   & $\ge 2$ & $\ge 1$ & $\ge 1$ & $0$                    \\\hline
\end{tabular}
\end{table}
\end{prop}
\begin{proof}
Looking at the symmetries of the Jacobi fields, we find that 
$$\langle e_3, N\rangle  \in (C^\infty)_{++} , \hspace{1em} \langle e_2, N\rangle \in (C^\infty)_{+-} , \hspace{1em} \langle e_1, N\rangle \in (C^\infty)_{-+} , \hspace{1em} \det(p,N,\lambda) \in (C^\infty)_{--}.$$
This gives the desired bounds on nullity. 

Using \cite[Corollary 5 and Proposition 2]{choe1990index} and the symmetry of the surface, we have that $\det(p,N,\lambda)$ is strictly positive on the interior of the quarter-bubble $\Lambda$. By Courant's nodal domain theorem \cite{courant1953methods}, this shows $\det(p,N,\lambda)$ is the first eigenfunction of the Jacobi operator on  $\Lambda$ with double Dirichlet boundary conditions, which implies $\Index_{--}(\Sigma_t) = \Index_{DD}(\Lambda) = 0$. Applying the above inequalities yields the other bounds on the indices.
\end{proof}
\begin{rem}
For the Costa surface $\Sigma_1$, the actual decomposition of the nullites and indices is given by

\begin{table}[!ht]
\begin{tabular}{|l|l|l|l|l|}\hline
$\Sigma_1$        & $++$                   & $+-$                   & $-+$                   & $--$                   \\\hline
nullity & $1 $ & $1$ & $ 1$ & $1$ \\\hline
index   & $3$ & $1$ & $1$ & 0                    \\\hline
\end{tabular}
\end{table}
which follows by examining the computations of Nayatani \cite[Section 5]{nayatani1993morse}.
\end{rem}

\bibliographystyle{alpha}
\bibliography{main}
\end{document}